\documentclass[11pt]{amsart}
\usepackage{amssymb}
\usepackage{amsfonts}
\usepackage{amsmath}
\usepackage{amsthm}

\usepackage[english]{babel}
\usepackage[latin1]{inputenc}
\usepackage{amsmath,amssymb,amsthm,epsfig}
\usepackage{eufrak}
\usepackage{color}
\usepackage{mathrsfs}
\usepackage{latexsym}

\theoremstyle{plain}

\newtheorem{teo}{Theorem}[section]

\newtheorem{defi}{Definition}[section]
\newtheorem{lema}{Lemma}[section]

\theoremstyle{definition}
\newtheorem{exe}{Example}

\begin{document}
\title{Engel theorem through singularities }
\author{M. Corr\^ea}
\address{M. Corr\^ea,\
ICEx - UFMG \\
Departamento de Matem\'atica \\
Av. Ant\^onio Carlos 6627 \\
30123-970 Belo Horizonte MG, Brazil
\\
} \email{mauricio@mat.ufmg.br}
\author{Luis. G. Maza}
\address{
Luis. G. Maza
\\Universidade Federal de Alagoas, Instituto de
Matem\' atica. Av. Lourival Melo Mota, s/n, Tabuleiro dos Martins
57072-970 - Maceio, AL-Brasil
\\
} \email{lmaza@im.ufal.br}
\date\today

\begin{abstract}
We prove a singular version of the Engel theorem.
We prove a normal form theorem for germs of holomorphic singular
Engel systems with good conditions on its singular set. As an
application, we prove that there exists an integral analytic curve
passing through the singular points of the system. Also, we prove
that a globally decomposable Engel system on a four dimensional
projective space has singular set with atypical codimension.
\end{abstract}

\maketitle

\section{Introduction}

A germ of \emph{ holomorphic Pfaff system} of codimension $k$ on
$(\mathbb{C}^n,0)$ is a subsheaf $\mathcal{I}$ of the cotangent
sheaf $\Omega^1_{\mathbb{C}^n}$ of $(\mathbb{C}^n,0)$ spanned by $k$
germs of holomorphic differential $1$-forms
$\omega_{1},\ldots,\omega_{k}$ assumed 
linearly independent at a generic point near 0. We will write
$\mathcal{I}=\langle\omega_{1},\ldots,\omega_{k}\rangle$. This
system can be represented by the holomorphic $k$-form
$\omega_{1}\wedge\ldots\wedge\omega_{k}.$ The singular set of
$\mathcal{I}$ is the analytic subset given by
$$
\mathrm{Sing}(\mathcal{I})=\{p\in (\mathbb{C}^n,0); \
(\omega_{1}\wedge\cdots \wedge\omega_{k})(p)=0\}.
$$
Therefore $\mathrm{Sing}(\mathcal{I})$ is defined by $k \times k$
determinants of an $n \times k$ matrix. Therefore, each irreducible
component has codimension at most $k + 1$. We say that the singular
set of $\mathcal{I}$ has \emph{expected codimension}  if it is a (startified) submanifold of $\mathbb{C}^n$ of codimension  $k+1$.

Let $\mathcal{C}=V(A)$ be a germ of analytic subset in
$(\mathbb{C}^n,0)$ of codimension $\leq k,$ with zeros ideal
$A$. If
$A=\langle f_1,\dots,f_r\rangle$, then we denote by
$dA$ the Pfaff system spanned by $df_1,\dots,df_r$.

By definition, we say that $\mathcal{C}=V(A)$ is an \emph{integral}
variety of $\mathcal{I}=\langle\omega_{1},\ldots,\omega_{k}\rangle$
 if
$$
\omega_{i} \wedge dA\in A\otimes
\Omega^{r+1}_{\mathbb{C}^n}, \ \mathrm{for}\ \mathrm{each} \
i=1,\dots,k.
$$

A Pfaff system is called integrable $\mathcal{I} $   if 
$$
d\mathcal{I} \equiv 0\mod\mathcal{I}.
$$
If $\mathcal{I} $ is integrable, by the classical Frobenius's
Theorem, for all points $p\in
(\mathbb{C}^n,0)\backslash \mathrm{Sing}(\mathcal{I})$ there exists
an integral complex analytic manifold of codimension $k$ passing
through $p$. In \cite{Malgrange} B. Malgrange obtained a
Frobenius's Theorem for singular integrable systems with singular
set of codimension $\geq 3,$ showing the existence of integral
varieties passing through the singular points of the system.

For a germ of Pfaff system $\mathcal{I}$, we can define its
\emph{derived flag}
$\mathcal{I}^{(0)}\supset\mathcal{I}^{(1)}\supset\cdots$ by the
relations $\mathcal{I}^{(0)}=\mathcal{I}$ and
$$\mathcal{I}^{(i+1)}=\{\alpha
\in\mathcal{I}^{(i)}: d \alpha \equiv 0 \mod\mathcal{I}^{(i)}\}.
$$
Then, the derived flag of a Pfaff system $\mathcal{I}$ is defined
inductively by the exact sequence
$$
0\longrightarrow \mathcal{I}^{(i+1)} \longrightarrow
\mathcal{I}^{(i)} \longrightarrow
d\mathcal{I}^{(i)}/\left(\mathcal{ I}^{(i)} d
\mathcal{I}^{(i)}\right)\longrightarrow 0.
$$
Since the codimension of each Pfaff system $\mathcal{I}^{(i)}$ is
constant on an open and dense set in a neighbourhood of $0$,  then there will be an integer $N$ such that
$\mathcal{I}^{(N)}=\mathcal{I}^{(N+1)}$. This integer $N$ is called
the \emph{derived length} of $\mathcal{I}$. 
Note that the Pfaff system
$\mathcal{I}^{(N)}$ is always integrable   since
$$
d\mathcal{I}^{(N)}\equiv 0\mod\mathcal{I}^{(N)}.
$$
If $\mathcal{I}^{(N)}=0$ we say that the system $\mathcal{I}$ is
\emph{completely nonholonomic}. See \cite{BCGGG} for more details.

Let $ \Theta_n$ be the sheaf of germs of holomorphic vector fields on $(\mathbb{C}^n,0) $.
A local $(r, n)$ holomorphic  distribution is the germ at $0\in \mathcal{C}^n$ of a rank $r$ subbundle of the tangent bundle $\Theta_n$

A Pfaff system $\mathcal{I}=\langle\omega_{1},\ldots,\omega_{k}\rangle$ of codimension $k$ induces  a  singular $(n-k, n)$ distribution  defined  by
$$
\mathcal{D}^0:=\mathrm{Ker}(\mathcal{I})=\{v \in \Theta_n;\ \omega_i(v)\equiv 0,\ \forall \ i\}.
$$
Thus, the induced distributions of the derived flag of $\mathcal{I}$ is given by
$$
\mathcal{D}^i:=\mathrm{Ker}(\mathcal{I}^i)= \mathrm{Ker}(\mathcal{I}^{i+1})+ [\mathrm{Ker}(\mathcal{I}^{i+1}), \mathrm{Ker}(\mathcal{I}^{i+1})],
$$
where $[\mathcal{D},\mathcal{D}]$ denotes the sheaf generated by Lie brackets
$[u,v ]$ with $u,v\in \mathcal{D}.$

A contact system on $(\mathbb{C}^3,0)$ is a completely nonholonomic
system. The classical Darboux-Pfaff theorem gives a normal form for
non-singular contact system. That is, if $\mathcal{I}$ is a
non-singular contact system, then there exist a germ of coordinate
system on $(\mathbb{C}^3,0)$ such that
$$
\mathcal{I}=\langle dz_3-z_2dz_1 \rangle.
$$
D. Cerveau showed in  \cite{C}  a  singular version of the Pfaff-Darboux theorem.

In this work we are interested in completely nonholonomic $(2, 4)$ holomorphic distributions 
with  derived length equal
to $2$.

\begin{defi}\label{defi Engel}\footnote{We would like to thank the anonymous referee for  suggest  us  this definition }
A local    \emph{  Engel distribution}  is a $(2, 4)$ holomorphic distribution  $\mathcal{D}\subset  \Theta_4$ satisfying the following conditions:
\begin{itemize}
\item[(i)] $\mathcal{D}^1= \mathcal{D} + [\mathcal{D} , \mathcal{D}]$ has rank three at a generic point near $0$.
\item[(ii)]  $\mathcal{D}^2= \mathcal{D}^1 + [\mathcal{D}^1 , \mathcal{D}^1]$ has rank four  at a generic point near $0$.
\end{itemize}
\end{defi}
Let $v_1,v_2\in \mathcal{D}$ germs on $(\mathbb{C}^4,x)$  of   generators  of the stalk  $ \mathcal{D}_x$. Then
$$\langle v_1,v_2, [v_1,v_2]\rangle=\mathcal{D}_{x}^1$$
and
$$\langle v_1,v_2, [v_1,v_2],[v_1,[v_1,v_2]], [v_2,[v_1,v_2]]\rangle=\mathcal{D}_{x}^2=\Theta_{4,x}.$$
Therefore, the singular set of $\mathcal{D}^1$ is set of  the dependence  of  $ v_1,v_2, [v_1,v_2]$ and 
the germ on $x$ of singular set of $\mathcal{D}^2$  is  the set of  dependence 
 $$ v_1\wedge v_2 \wedge [v_1,v_2]\wedge[v_1,[v_1,v_2]]\wedge [v_2,[v_1,v_2]]=0.$$
Moreover, the distribution $\mathcal{D}^1$ is induced by the $1$-form
$$
\beta=i_{v_1} i_{v_2}i_{[v_1,v_2]} (dz_1\wedge dz_2 \wedge dz_3\wedge dz_4 ).
$$
To the distribution $\mathcal{D}$ we associate the so called Engel vector field $Y$ defined by
$$
i_Y  (dz_1\wedge d_2 \wedge dz_3\wedge dz_4 )=\beta \wedge d\beta.
$$
In particular, $\mathrm{Sing}(\mathcal{D}^1)= \mathrm{Sing}(\beta)\subset  \mathrm{Sing}(Y)$. We will call of \textit{characteristic Engel foliation} of $\mathcal{D}$
the rank one vector bundle $\mathcal{L}(\mathcal{D})\subset \mathcal{D}$ induced by the Engel vector fields.

 In this work our main theorem is the following.
\begin{teo}\label{main}
Let $\mathcal{D}$ be a germ of holomorphic
Engel distribution on $(\mathbb{C}^4,0)$ with
$\mathrm{codim}(\mathrm{Sing}(\mathcal{D}))\geq 2$ and
$\mathrm{codim}(\mathrm{Sing}(\mathcal{L}(\mathcal{D}))\geq 3$. Then there
exist $f_1,\dots,f_4\in\mathcal{O}_0^4$ such the distribution $\mathcal{D}$ is induced by 
the system
$$
\langle df_4-f_3df_1, df_3-f_2df_1\rangle.
$$
\end{teo}

This result is consequence of a theorem on normal form of Pfaff system of derived length
equal to $2$, see Theorem \ref{teo}. In \cite{Ge} and \cite[Appendix]{M} is showed that there exist a pair of $1$-form $(\alpha,\beta)$ inducing a regular  Engel distribution  $\mathcal{D}$ satisfying
 \begin{itemize}
\item[(i)] $\alpha\wedge\beta \wedge d\alpha \not\equiv 0$
\item[(ii)] $\alpha\wedge\beta \wedge d\beta\equiv 0$
\item[(iii)] $\beta \wedge d\beta \not\equiv 0 $.
\end{itemize}

This motivates  us  the following definition.
\begin{defi}\label{defi Engel}
A germ of \emph{ singular Engel system} in $(\mathbb{C}^4,0)$ is a Pfaff
system $\mathcal{I} $, of codimension $2$,  which can be described by $1$-forms $\alpha$ and $\beta $ satisfying the following conditions:
\begin{itemize}
\item[(i)] $\alpha\wedge\beta \wedge d\alpha \not\equiv 0$
\item[(ii)] $\alpha\wedge\beta \wedge d\beta\equiv 0$
\item[(iii)] $\beta \wedge d\beta \not\equiv 0 $,
\end{itemize}
\end{defi}
We can see that a germ of singular \emph{Engel system} $\mathcal{I} $  in $(\mathbb{C}^4,0)$
is a system of codimension $2$ such that, for $0 \leq i \leq 2$,
the elements of its derived flag satisfy
$\mathrm{codim}(\mathcal{I}^{(i)})=2-i$. In fact, $\mathcal{I}^{(0)}=\langle
\alpha,\beta \rangle$, $\mathcal{I}^{(1)}=\langle \beta \rangle$
and $\mathcal{I}^{(2)}=0$. Thus, an Engel system has derived length
equal to $2$. Moreover, we have that $d\mathcal{I}^{(1)}=\langle d\beta \rangle$. In this case,
we define the singular set of $d\mathcal{I}^{(1)}$ by
$$
\mathrm{Sing}(d\mathcal{I}^1):=\{ p\in \mathrm{Sing}(\mathcal{I}^1) ;\ d\beta(p)=0 \}
$$
Observe that $ \mathrm{Sing}(d\mathcal{I}^1)$ is well defined. In fact, let $\delta=g \beta $ , for some $g\in \mathcal{O}_0^*$. Here $ \mathcal{O}_0^*$ germs of nowhere vanishing holomorphic functions.
Then, if $p\in \mathrm{Sing}(\mathcal{I}^1)$ then $d\delta(p)=g(p) d\beta(p)+\beta(p)\wedge dg(p)=g(p) d\beta(p)$. We conclude that $d\delta(p)=0$ if and only if $d\beta(p)=0$ for all $p\in \mathrm{Sing}(\mathcal{I}^1)$, since $g\in \mathcal{O}_0^*$.

\begin{exe}
The conditions of Theorem \ref{main} are  necessary. Consider the family of Engel distribution given by 
$$
\mathcal{D}_r=\mathrm{Ker}(\langle dz_4-z_3^rdz_1, dz_3-z_2dz_1 \rangle),
$$
with $r\geqslant 2$.
Set $\beta_r:=dz_4-z_3^rdz_1$ and $\alpha_r= dz_3-z_2dz_1$. A calculation shows that the pair of $1$-forms $(\alpha,\beta)$ satisfy the conditions $i)$,$ii)$ and $iii)$ of definition \ref{defi Engel} and the following hold
\begin{equation}\label{exemplo}
\beta_r \wedge d\beta_r=-rz_3^{r-1} dz_4\wedge dz_3\wedge dz_1 
\end{equation}
and
\begin{equation}\label{exemplo2}
\beta_r \wedge \alpha_r  \wedge d\alpha_r= -dz_4\wedge dz_3\wedge dz_2\wedge dz_1.
\end{equation}
For each $r$, the Engel distribution  $\mathcal{D}_r$
is regular but $\mathrm{codim}(\mathrm{Sing}(\mathcal{L}(\mathcal{D}))=1$, since 
$$
\mathrm{Sing}(\mathcal{L}(\mathcal{D}_r))=\mathrm{Sing}(\beta_r \wedge d\beta_r)=\{z_3^{r-1}=0\}.
$$
On the other hand, the  distribution $\mathcal{D}_r$ can not be described by  two differential $1$-forms of the form
$$
\langle df_4-f_3df_1, df_3-f_2df_1\rangle.
$$
Let us suppose by absurd that  $\beta_r=df_4-f_3df_1$ and $\alpha_r= df_3-f_2df_1$. In this case, we have that
\begin{center}
$\beta_r \wedge d\beta_r=-df_3\wedge df_2\wedge df_1$ and  
$\beta_r \wedge \alpha_r  \wedge d\alpha_r= -df_4\wedge df_3\wedge df_2\wedge df_1$.
\end{center}
This implies that 
\begin{center}  
$\beta_r \wedge \alpha_r  \wedge d\alpha_r= df_4\wedge \beta_r \wedge d\beta_r $.
\end{center}
Substituting (\ref{exemplo}) we obtain 
\begin{center}  
$\beta_r \wedge \alpha_r  \wedge d\alpha_r= df_4\wedge \beta_r \wedge d\beta_r =df_4\wedge(-rz_3^{r-1} dz_4\wedge dz_3\wedge dz_1)$.
\end{center}
Thus, 
\begin{center}  
$ \beta_r \wedge \alpha_r  \wedge d\alpha_r=
-r\frac{\partial f_4}{\partial z_2}z_3^{r-1} dz_4\wedge dz_3 \wedge dz_2\wedge dz_1$.
\end{center}
An absurd, since it follows from (\ref{exemplo2}) that the $4$-form $\beta_r \wedge \alpha_r  \wedge d\alpha_r$ is not singular. 
\end{exe}

In the real non-singular case, these Pfaff systems were introduced
by E. von Weber in 1898 and studied by several authors
\cite{Cartan}\cite{GKR}\cite{PR}. F. Engel \cite{Engel} shows that a
non-singular generic Engel system is locally isomorphic, at a generic point,
to the canonical system
\begin{equation}
\label{canonical}
\mathcal{I}_0=\langle dz_4-z_3dz_1, dz_3-z_2dz_1 \rangle.
\end{equation}
That is, F. Engel provides a kind of Pfaff-Darboux type theorem for
non-singular Pfaff systems of codimension $2$ and derived length
equal to $2$. M. Zhitomirskii in \cite{Z} obtained normal forms for
real non-singular Engel along non generic points.

The canonical system appears naturally as a system called canonical
contact system on the space $J^{2}(\mathbb{C},\mathbb{C})$ of
$2$-jets of holomorphic maps of $\mathbb{C}$, see \cite{PR}.
Nonsingular global holomorphic Engel systems have been studied by
L. Sol\'a Conde and F. Presa in \cite{PS}.

We prove the following result for germs of holomorphic Engels system
in $(\mathbb{C}^4,0)$
\begin{teo}\label{teo}
Let $\mathcal{I}=\langle\alpha,\beta\rangle$ be a germ of holomorphic
Engel system on $(\mathbb{C}^4,0)$ with
$\mathrm{codim}(\mathrm{Sing}(\mathcal{I}))\geq 2$ and
$\mathrm{codim}(\mathrm{Sing}(d\beta))\geq 3$. Then there
exist $f_1,\dots,f_4\in\mathcal{O}_0^4$ such
that
$$
\mathcal{I}=\langle df_4-f_3df_1, df_3-f_2df_1\rangle.
$$
More precisely, there exists a germ of holomorphic map
$f:=(f_1,f_2,f_3,f_4): (\mathbb{C}^4,0)\circlearrowleft$ which is a
biholomorphism outside $\mathrm{Sing}(\mathcal{I})\cup
\mathrm{codim}(\mathrm{Sing}(d\beta))$ such that
$f^*\mathcal{I}_0=\mathcal{I}$.
\end{teo}



Normal forms allows us to prove the existence of integral
submanifolds. An interesting consequence of Theorem \ref{teo} is the
existence of an integral analytic curve passing through the
singular points of the system.
\begin{teo} Let $\mathcal{I}=\langle\alpha,\beta\rangle$ be a germ of holomorphic
Engel system on $(\mathbb{C}^4,0)$ with
$\mathrm{codim}(\mathrm{Sing}(\mathcal{I}))\geq 2$ and
$\mathrm{codim}(\mathrm{Sing}(d\beta))\geq 3$, then there
exists a germ of an analytic curve passing through
$\mathrm{Sing}(\mathcal{I})$ which is a solution of $\mathcal{I}$.
\end{teo}

\begin{proof}
In fact, it follows from Theorem \ref{teo} that the analytic curve
$\{f_1=f_3=f_4=0\}$ is a solution of $\mathcal{I}=\langle
df_4-f_3df_1, df_3-f_2df_1\rangle.$
\end{proof}

Finally, we give Another application of Theorem \ref{teo} for
globally decomposable Engel system on four dimensional projective
space.

It is well known that all codimension one integrable systems in
$\mathbb{P}^n$ have in its singular set an irreducible component of
codimension two.
\begin{teo}\cite{Jouanolou} \cite{Alcides}
Let $\mathcal{I}$ be a codimension one integrable system on
$\mathbb{P}^n$, $n\geq 3$, such that
$\mathrm{codim}(\mathrm{Sing}(\mathcal{I}))\geq 2$. Then
$\mathrm{Sing}(\mathcal{I})$ has an irreducible component of
codimension two.
\end{teo}

Let $\mathcal{I}$ be a codimension one integrable system on a Fano
manifold such that $ \mathrm{Sing}(\mathcal{I})\neq \emptyset$. In
\cite[Corollary 4.7 ]{LPT} F. Loray, J. V. Pereira and F Touzet show
that if the canonical class of $\mathcal{I}$ is numerically trivial
then its singular set has a component of codimension two.

A similar situation appears in the study of singularities of Poisson
structures on Fano manifolds motivated by Bondal's conjecture
\cite{Bondal} \cite[Conjecture 4]{Be}. A. Polishchuk in \cite{Po}
showed that the rank of a nondegenerate Poisson structure on a Fano
variety of odd dimension drops along a subset of codimension two.

As an application of Theorem \ref{teo} we prove that a globally
decomposable Engel system on four dimensional projective space has a
singular set with atypical codimension. In fact, the expected
codimension of the singular set of a Pfaff system of codimension
$2$ should be $3$. But, the following Theorem shows that the
singular set of these systems has codimension $\leq 2.$

\begin{teo}
Let $\mathcal{I}$ be a globally decomposable holomorphic Engel
system on $\mathbb{P}^4$. Then, either
$\mathrm{Sing}(d\mathcal{I}^{(1)})$ has a component of codimension
two, or $\mathrm{Sing}(\mathcal{I})$ has a component of codimension
one. Moreover, if $\mathrm{codim}(\mathrm{Sing}(\mathcal{I}))\geq 2$,
then $\mathrm{Sing}(\mathcal{I})$ has a component of codimension
two.
\end{teo}

\section{Proof of the Theorem \ref{teo}}
\begin{proof}
Since $d\beta\wedge \beta\not\equiv 0$ and $(d\beta)^2\wedge \beta\equiv
0$, we have that $\beta$ has class $1$. By Cerveau's singular version of the
Pfaff-Darboux Theorem \cite{C} we get that there exist $f_1,f_3,f_4\in\mathcal{O}_0^4$
such that
$$
\beta=df_4-f_3df_1.
$$
In particular, $d\beta=df_1\wedge df_3$. Now, since
$d\beta\wedge\alpha\wedge\beta\equiv 0$ we get
$$
0=d\beta\wedge\alpha\wedge\beta=df_1\wedge
df_3\wedge\alpha\wedge\beta.
$$
This implies that there exist germs of holomorphic funtions
$\widetilde{a},\widetilde{b}$ and $\widetilde{\lambda}$ on $U=
(\mathbb{C}^4,0)-Sing(\alpha \wedge \beta)\cup Sing(d\beta) $ such
that
$$
\alpha|_{U}=\widetilde{a}df_1+\widetilde{b}df_3+\widetilde{\lambda}\beta.
$$
Since the codimension of $\mathrm{Sing}(\alpha \wedge \beta)\cup
\mathrm{Sing}(d\beta)$ is bigger than $2$, by Hartorgs' extension
Theorem we have the identity
$$
\alpha - \lambda\beta=adf_1+bdf_3, \ \mathrm{where} \ a,b,\lambda
\in \mathcal{O}_0^4
$$
on $(\mathbb{C}^4,0).$ Now if either $a = 0$ or $b= 0$, then
$\alpha\wedge\beta \wedge d\alpha\equiv 0$, a contradiction to
Engel's conditions. Thus $a \neq 0$ and $b \neq 0$, and therefore
$$
\frac{1}{b}\alpha - \frac{\lambda}{b}\beta=\frac{a}{b}df_1+df_3
$$
and if we set $f_2=-\frac{a}{b}$ then
$$
\frac{1}{b}\alpha - \frac{\lambda}{b}\beta=df_3-f_2df_1.
$$
Thus,
$$
\mathcal{I}=\langle\alpha,\beta\rangle=
\left\langle\alpha,\frac{1}{b}\alpha -
\frac{\lambda}{b}\beta\right\rangle=\langle
df_4-f_3df_1,df_3-f_2df_1\rangle.
$$
Now, we will prove that the map
$f:(:(\mathbb{C}^4,0)\circlearrowleft$ defined by $f_1,f_2,f_3,f_4)$
is a biholomorphism outside $Sing(\alpha \wedge \beta)\cup
Sing(d\beta)$. That is, we prove that
$$
df_1\wedge df_2 \wedge df_4\wedge df_3
$$
never vanishes outside $\mathrm{Sing}(\alpha \wedge \beta)\cup
\mathrm{Sing}(d\beta)$. Differentiating the identity
$$
\frac{1}{b}\alpha=\frac{a}{b}df_1+df_3+\frac{\lambda}{b}\beta.
$$
we get
$$
d\left(\frac{1}{b}\right)\wedge\alpha+\frac{1}{b}d\alpha=d\left(\frac{a}{b}\right)\wedge
df_1+d\left(\frac{\lambda}{b}\right)\wedge\beta+\frac{\lambda}{b}d\beta.
$$
Multiplying this identity by $\beta\wedge\alpha$ we obtain
$$
\left[d\left(\frac{a}{b}\right)\wedge
df_1+d\left(\frac{\lambda}{b}\right)\wedge\beta+\frac{\lambda}{b}d\beta
\right]\wedge\beta\wedge\alpha=d\left(\frac{a}{b}\right)\wedge
df_1\wedge \beta\wedge\alpha
$$
since $d\beta\wedge\beta\wedge\alpha\equiv 0$. Thus
$$
d\left(\frac{a}{b}\right)\wedge df_1\wedge \beta\wedge\alpha
=\left[d\left(\frac{1}{b}\right)\wedge \alpha +\frac{1}{b}d\alpha
\right]\wedge \beta\wedge\alpha.
$$
Therefore
$$
d\left(\frac{a}{b}\right)\wedge df_1\wedge \beta\wedge\alpha
=\frac{1}{b}d\alpha \wedge\beta\wedge\alpha \neq 0.
$$
Using that $\alpha =adf_1+bdf_3+\lambda\beta$ and
$\beta=df_4-f_3df_1$ and substituting in
$d\left(\frac{a}{b}\right)\wedge df_1\wedge \beta\wedge\alpha $ we
conclude that
$$
0\neq d\left(\frac{a}{b}\right)\wedge df_1\wedge
\beta\wedge\alpha=bd\left(\frac{a}{b}\right)\wedge df_1\wedge
df_4\wedge df_3=bdf_2\wedge df_1\wedge df_4\wedge df_3
$$
is nowhere vanishing outside $\mathrm{Sing}(\alpha \wedge
\beta)\cup \mathrm{Sing}(d\beta).$
\end{proof}

\section{Proof of the Theorem \ref{main}}

It follows from  \cite{Ge} and \cite[Appendix]{M} that  there exist a pair of $1$-form $(\alpha,\beta)$ inducing an   Engel distribution  $\mathcal{D}$ on $\mathbb{C}^4-( \mathrm{Sing}(\mathcal{D})\cup  \mathrm{Sing}(\mathcal{D}^1))$ satisfying
 \begin{itemize}
\item[(i)] $\alpha\wedge\beta \wedge d\alpha \not\equiv 0$
\item[(ii)] $\alpha\wedge\beta \wedge d\beta\equiv 0$
\item[(iii)] $\beta \wedge d\beta \not\equiv 0 $.
\end{itemize}
Since   $\mathrm{codim}(\mathrm{Sing}(\mathcal{D})\cup  \mathrm{Sing}(\mathcal{D}^1))\geqslant 2$
then by Hartogs's extension theorem we can extend the pair $(\alpha,\beta)$ inducing  the distribution $\mathcal{D}$.  Therefore $\mathcal{D}$ is induced by the Engel system $<\alpha,\beta>$.
Thus, the Engel vector field $\mathcal{L}(\mathcal{D})$ is induced by 
$$
\beta \wedge d\beta.
$$
In particular , $\mathrm{Sing}(\mathcal{L}(\mathcal{D}))=\mathrm{Sing}(\beta \wedge d\beta)\supseteq \mathrm{Sing}(d\beta)$. Then   $$\mathrm{codim}(\mathrm{Sing}(d\beta))\geqslant \mathrm{codim}(\mathrm{Sing}(\mathcal{L}(\mathcal{D}))\ \geqslant 3.$$
The result follows from Theorem \ref{teo}.

\section{Application to Engel systems on projective spaces}
A \emph{Pfaff system $\mathcal{I}$ of codimension $k$} on a complex
projective space $\mathbb{P}^n$ is a locally decomposable section
$$
\omega_{\mathcal{I}} \in
H^0(\mathbb{P}^n,\Omega^{k}_{\mathbb{P}^n}\otimes\mathcal{L}).
$$
This means that for all $p\in \mathbb{P}^n$ there exists a
neighborhood $U$ of $p$ and $1$-forms $\omega_1,\ldots ,\omega_k\in
\Omega^{1}_{U}$, such that $\omega_{\mathcal{I}}|_{U}=\omega_1\wedge
\cdots \wedge\omega_k.$

If $i: \mathbb P^k \to \mathbb P^n$ is a generic linear immersion
then $i^* \omega_{\mathcal{I}} \in H^0(\mathbb P^k,
\Omega^k_{\mathbb P^k} \otimes \mathcal L)$ is a section of a line
bundle, and its divisor of zeros reflects the tangencies between
$\mathcal{I}$ and $i(\mathbb P^k)$. The \emph{degree} of
$\mathcal{I}$ is, by definition, the degree of such a tangency
divisor. Set $d:=\deg(\mathcal{I})$. Since $\Omega^k_{\mathbb
P^k}\otimes \mathcal L = \mathcal O_{\mathbb P^k}( \deg(\mathcal L)
- k - 1)$, one concludes that $\mathcal L= \mathcal O_{\mathbb
P^n}(d+ k + 1)$.

We say that $\mathcal{I}$ is \emph{globally decomposable } if
$\omega_{\mathcal{I}}=\omega_1\wedge \cdots \wedge\omega_k$ for
suitable $\omega_i\in
H^0(\mathbb{P}^n,\Omega^{k}_{\mathbb{P}^n}\otimes\mathcal{L}_i).$

Besides, the Euler sequence implies that a section $\omega$ of
$\Omega^k_{\mathbb P^n} ( d + k + 1 )$ can be thought of as a
polynomial $k$-form on $\mathbb{C}^{n+1}$ with homogeneous
coefficients of degree $d + 1$, which we will still denote by
$\omega$, satisfying
\begin{equation}
\label{equirw} i_R \omega = 0
\end{equation}
where
$$
R=x_0 \frac{\partial}{\partial x_0} + \cdots + x_n
\frac{\partial}{\partial x_n}
$$
is the radial vector field. Thus the study of distributions of
degree $d$ on $\mathbb P^n$ reduces to the study of locally
decomposable homogeneous $k$-forms of degree $d+1$ on $\mathbb
C^{n+1}$ satisfying the relation (\ref{equirw}).

We will use the following Jouanolou's Lemma.

\begin{lema} \cite[Lemme 1.2, pp. 3]{Jouanolou}\label{Jouanolou}
If $\eta$ is a homogeneous $q$-form of degree $s$, then
$$
i_R d\eta + d( i_R \eta) = (q+s) \eta
$$
where $R$ is the radial vector field and $i_R$ denotes the interior
product or contraction with $R$.
\end{lema}

As an application of Theorem \ref{teo} we prove that a globally
decomposable Engel system on four dimensional projective space has a
singular set with atypical codimension. In fact, the expected
codimension of the singular set of a codimension $2$ should be $3$.
But, the following Theorem shows that the singular set of these
systems has codimension $\leq 2.$

\begin{teo}
Let $\mathcal{I}$ be a globally decomposable holomorphic Engel
system on $\mathbb{P}^4$. Then , either
$\mathrm{Sing}(d\mathcal{I}^{(1)})$ has a component of codimension
two, or $\mathrm{Sing}(\mathcal{I})$ has a component of codimension
one. Moreover, if $\mathrm{codim}(\mathrm{Sing}(\mathcal{I}))\geq 2$,
then $\mathrm{Sing}(\mathcal{I})$ has a component of codimension
two.
\end{teo}

\begin{proof}
Firstly we observe that on $\mathbb{P}^4$ all holomorphic
$1$-forms $\beta \in H^0(\mathbb{P}^4,\Omega^1_{\mathbb{P}^4}(s))$
satisfy $\beta\wedge (d\beta)^2 \equiv 0,$ since $\beta\wedge
(d\beta)^2$ is a $5$-form on $\mathbb{P}^4$. Moreover, we have that
$\mathrm{Sing}(d\mathcal{I}^{(1)})=\mathrm{Sing}(d\beta).$
Indeed, by lemma
\ref{Jouanolou} we have
\begin{equation}\label{Joua}
i_R(d\beta)=(s+1) \beta
\end{equation}
since $i_{R}\beta=0$. Thus, $\mathrm{Sing}(d\beta)\subset \mathrm{Sing}(\beta)$ and this implies that
$$
\mathrm{Sing}(d\mathcal{I}^{(1)}):=\mathrm{Sing}(d\beta)\cap \mathrm{Sing}(\beta)=\mathrm{Sing}(d\beta).
$$

Now, suppose that $\mathrm{codim}(\mathrm{Sing}(\mathcal{I}))\geq 2$ and
$\mathrm{codim}(\mathrm{Sing}(d\mathcal{I}^{(1)})) \geq 3$. Then, it
follows from Theorem \ref{teo} that there exist homogeneous
polynomials $f_1,f_2,f_3,f_4$ on $\mathbb{C}^5$ such that
$$
\alpha=df_4-f_3df_1, \beta=df_3-f_2df_1.
$$
In particular, we have that $\beta\wedge d\beta=-df_4\wedge
df_3\wedge df_1$ and $\alpha \wedge \beta\wedge d\alpha=df_1\wedge
df_2\wedge df_3\wedge df_4$.

Since $i_{R}\alpha=i_{R}\beta=0$, we conclude that
$k_4f_4-k_1f_3f_1=k_3f_3-k_1f_2f_1=0$, where $k_i=\deg(f_i)$,
$i=1,3,4.$ These relations imply that
$$
\beta\wedge d\beta=\alpha \wedge \beta\wedge d\alpha\equiv 0.
$$
This is a contradiction.
On the other hand, suppose that
$\mathrm{codim}(\mathrm{Sing}(\mathcal{I}))\geq 2$. Using the relation (\ref{Joua}) we have that
$$
\mathrm{Sing}(d\beta)\subset\mathrm{ Sing}(i_R(d\beta)\wedge
\alpha)=\mathrm{Sing}((s+1)\beta\wedge
\alpha)=\mathrm{Sing}(\beta\wedge \alpha).
$$
We conclude that $\mathrm{Sing}(\mathcal{I})$ has a component of
codimension two.
\end{proof}

\begin{exe}
Consider the differential system induced by the $1$-forms
$$\alpha=z_{0}^{2}dz_{4}-z_{0}z_{3}dz_{1}+(z_{1}z_{3}-z_{0}z_{4})dz_{0}$$
and
$$\beta=z_{0}^{2}dz_{3}-z_{0}z_{2}dz_{1}+(z_{1}z_{2}-z_{0}z_{3})dz_{0}.$$

A calculation shows that the pair of $1$-forms $(\alpha,\beta)$
satisfy the conditions $i)$,$ii)$ and $iii)$ of definition
\ref{defi Engel} and $i_{R}\alpha=i_{R}\beta=0$. Therefore, the
differential system $\mathcal{I}=\langle\alpha,\beta\rangle$ induces
a decomposable Engel system on $\mathbb{P}^4$.
We have that
$$\alpha \wedge\beta=z_{0}^{4}dz_{4}\wedge
dz_{3}-z_{0}^{3}z_{2}dz_{4}\wedge
dz_{1}+z_{0}^{2}(z_{1}z_{2}-z_{0}z_{3})dz_{4}\wedge dz_{0}-
$$
$$
-z_{0}^{3}z_{3}dz_{1}\wedge
dz_{3}-z_{0}z_{3}(z_{1}z_{2}-z_{0}z_{3})dz_{1}\wedge dz_{0}+
$$
$$
+z_{0}^{2}(z_{1}z_{3}-z_{0}z_{4})dz_{0}\wedge
dz_{3}-z_{0}z_{2}(z_{1}z_{3}-z_{0}z_{4})dz_{0}\wedge dz_{1}.
$$
Therefore $\mathrm{Sing}(\alpha \wedge\beta)=\{z_0=0\}$ has
codimension one. Moreover,
$$
\mathrm{Sing}(d\beta)=\{z_0=z_1=z_2=0\}
$$
has codimension $3$.
\end{exe}



\noindent{\footnotesize \textsc{Acknowlegments.} We are grateful to
Arturo Fernendez-Perez, Rogerio Mol, Marcio Soares and Israel
Vainsencher for pointing out corrections in previous versions of
this paper. We thank the referee for for kindly pointing out several
suggestions and corrections.

}

\end{document}